\documentclass[10pt]{amsart}
\usepackage{amssymb, amscd, amsmath, amsthm, amscd, curves,
epsf, epsfig, latexsym, enumerate, pstricks,graphics}

\newtheorem{theorem}{Theorem}
\newtheorem{lemma}[theorem]{Lemma}
\newtheorem{cor}[theorem]{Corollary}
\newtheorem*{ex}{Example}

\begin{document}
\title{3-manifolds with abelian embeddings in $S^4$}
 
\author{J.A.Hillman}
\address{School of Mathematics and Statistics\\
     University of Sydney, NSW 2006\\
      Australia }

\email{jonathan.hillman@sydney.edu.au}

\begin{abstract}
We consider embeddings of 3-manifolds in $S^4$ such that each of the two
complementary regions has an abelian fundamental group.
In particular, we show that an homology handle $M$
has such an embedding if and only if $\pi_1(M)'$ is perfect,
and that the embedding is then essentially unique.
\end{abstract}



\maketitle

Every integral homology 3-sphere embeds as a
topologically locally flat hypersurface in $S^4$, 
and has an essentially unique ``simplest" such embedding,
with contractible complementary regions.
For other 3-manifolds which embed in $S^4$, 
the complementary regions cannot both be simply-connected,
and it not clear whether they always have canonical ``simplest" embeddings.
If $M$ is a closed hypersurface in $S^4=X\cup_MY$ then
$H_1(M;\mathbb{Z})\cong{H_1(X;\mathbb{Z})}\oplus{H_1(Y;\mathbb{Z})}$,
and so embeddings such that each of the complementary regions 
$X$ and $Y$ has an abelian fundamental group
might be considered simplest.
We shall say that such an embedding is abelian.
Although most 3-manifolds that embed in $S^4$ do not have such embeddings, 
this class is of particular interest as the possible groups are known,
and topological surgery in dimension 4 is available for abelian fundamental groups.

Homology 3-spheres have essentially unique abelian embeddings
(although they may have other embeddings).
This is also known for $S^2\times{S^1}$ and $S^3/Q(8)$,
by results of Aitchison (published in \cite{Ru}) and Lawson \cite{La},
respectively.
In Theorems \ref{neutral} and \ref{homhandle} below we show that 
if $M$ is an orientable homology handle
(i.e., such that $H_1(M;\mathbb{Z})\cong\mathbb{Z}$)
then it has an abelian embedding if and only if $\pi_1(M)$ 
has perfect commutator subgroup, 
and then the abelian embedding is essentially unique.
(There are homology handles which do not embed in $S^4$ at all!)
The 3-manifolds obtained by 0-framed surgery 
on 2-component links with unknotted components
always have abelian embeddings, 
and the complementary regions for such embeddings 
are homotopy equivalent to standard 2-complexes.
These shall be our main source of examples.
In particular,
we shall give an example in which $X\simeq{Y}\simeq{S^1\vee{S^2}}$,
but the pairs  $(X,M)$ and $(Y,M)$ are not homotopy equivalent.
We do not yet have examples of a 3-manifold with several 
inequivalent abelian embeddings.

The first two sections fix our notation, recall some earlier work,
and give some results on homotopy equivalences. 
In \S3 we define the notions of abelian and nilpotent embeddings.
Sections 4--6 consider abelian embeddings of 3-manifolds $M$
with torsion free homology (i.e., $H_1(M;\mathbb{Z})\cong\mathbb{Z}^\beta$, 
where $\beta\leq4$ or $\beta=6$).
In \S7 we consider briefly some embeddings of rational homology spheres.
In particular, we shall show that ten 3-manifolds 
with elementary amenable fundamental group have abelian embeddings.
The question remains open for one further such 3-manifold.
In \S8 we give some simple observations on the possible homotopy types 
of the complementary regions for the final class of abelian embeddings,
for which  
$\pi_1(X)\cong\pi_1(Y)\cong\mathbb{Z}\oplus(\mathbb{Z}/k\mathbb{Z})$,
for some $k>1$.
While such examples do exist,
much less is known in this case.

All 3-manifolds considered here shall be closed, connected and orientable.
An embedding $j$ is {\it smoothable\/} if it is smooth with respect
to some smooth structure on $S^4$, 
equivalently, if each  complementary region is a handlebody.
Although the embeddings that we shall construct are 
usually smooth embeddings in the standard 4-sphere,
we wish to apply surgery arguments,
 and so ``embedding" shall mean ``topologically locally flat
embedding", unless otherwise qualified.
Embeddings $j$ and $\tilde{j}$ are {\it equivalent\/}
if there are self-homeomorphisms $\phi$ of $M$ and $\psi$ of $S^4$ such
that $\psi{j}=\tilde{j}\phi$.
If all abelian embeddings are equivalent to $j$,
we shall say that $j$ is {\it essentially unique}.

\section{notation and background}

Let $j:M\to{S^4}$ be an embedding of a closed connected 3-manifold,
and let $X$ and $Y$ be the closures of the components of $S^4\setminus{M}$.
The Mayer-Vietoris sequence for $S^4=X\cup_MY$ and Poincar\'e-Lefshetz
duality give isomorphisms
$H_i(M;\mathbb{Z})\cong{H_i(X;\mathbb{Z})}\oplus{H_i(Y;\mathbb{Z})}$ 
for $i=1$ and 2, $H_2(X;\mathbb{Z})\cong{H^1(Y;\mathbb{Z})}$ and  
$H_2(Y;\mathbb{Z})\cong{H^1(X;\mathbb{Z})}$,
while $H_i(X;\mathbb{Z})=H_i(Y;\mathbb{Z})=0$ for $i>2$.
Since $\chi(X)+\chi(Y)=\chi(S^4)+\chi(M)=2$,
we may assume that $\chi(X)\leq1\leq\chi(Y)$.
Let $\beta=\beta_1(M;\mathbb{Z})$, $\pi=\pi_1(M)$,
$\pi_X=\pi_1(X)$ and $\pi_Y=\pi_1(Y)$,
and let $j_X$ and $j_Y$ be the inclusions of $M$ into $X$ and $Y$, 
respectively.

Our commutator convention is that if $G$ is a group and $g,h\in{G}$ then $[g,h]=ghg^{-1}h^{-1}$.
The commutator subgroup is $G'=[G,G]$, and the second derived group is $G''=[G',G']$.
The lower central series is defined by $G_{[1]}=G$ 
and $G_{[n+1]}=[G,G_n]$ for all $n\geq1$.
Let $F(r)$ be the free group of rank $r$.

If $V$ is a cell-complex we shall write 
$C_*(\widetilde{V})=C_*(V;\mathbb{Z}[\pi_1(V)])$
for the cellular chain complex of the universal cover $\widetilde{V}$ 
with its natural structure as a $\mathbb{Z}[\pi_1(V)]$-module
(and similarly for pairs of spaces).

Our examples may all be constructed using bipartedly slice links.
Let $M(L)$ be the closed 3-manifold obtained by 0-framed surgery 
on the link $L$. 
We say that $L$ is {\it bipartedly slice\/} (respectively, {\it trivial} or {\it ribbon}) 
if it has a partition $L=L_+\cup{L_-}$ into two sublinks 
which are each slice links (respectively, trivial or ribbon links).
The partition then determines an embedding $j_L:M\to{S^4}$,
given by ambient surgery on an equatorial $S^3$ in $S^4=D_+\cup{D_-}$.
We add 2-handles to these 4-balls along $L_+$ on one side 
and along $L_-$ on the other.
If $L_+$ and $L_-$ are smoothly slice then $j_L$ is smooth, 
and if they are trivial each complementary region may be
obtained by adding 1- and 2-handles to the 4-ball.
(The notation $j_L$  is ambiguous, 
for if $L$ has more than one component it may have 
several different partitions leading to distinct embeddings.
Moreover we must choose a set of slice discs for each of $L_+$ and $L_-$.)
If each  complementary region for an embedding $j$
may be obtained from the 4-ball by adding 1- and 2-handles,
must $j=j_L$ for some 0-framed link $L$?

In \cite{Hi17} we said that an embedding $j$ is {\it minimal\/} if the
induced homomorphism $j_\Delta:\pi\to\pi_1(X)\times\pi_1(Y)$ is an epimorphism.
In fact this is equivalent to each of $j_X$ and $j_Y$ inducing an epimorphism.

\begin{lemma} 
\label{minimal}
The homomorphisms $j_{X*}=\pi_1(j_X)$ and $j_{Y*}=\pi_1(j_Y)$ 
are both epimorphisms if and only if 
$j_\Delta=(j_{X*},j_{Y*})$ is an epimorphism.
\end{lemma}

\begin{proof}
Let $K_X=\mathrm{Ker}(j_{X*})$ and $K_Y=\mathrm{Ker}(j_{Y*})$.
If $j_{X*}$ and $j_{Y*}$ are epimorphisms then they induce isomorphisms
$\pi/K_X\to\pi_X$ and $\pi/K_Y\to\pi_Y$.
Hence $\pi/K_XK_Y\cong\pi_X/j_{X*}(K_Y)$ and $\pi/K_XK_Y\cong\pi_Y/j_{Y*}(K_X)$.
Since  $\pi_1(X\cup_MY)=1$, these quotients must all be trivial.
If $g\in{K_X}$ and $h\in{K_Y}$ then $j_\Delta(gh)=(j_{X*}(h),j_{Y*}(g))$.
Hence $j_\Delta$ is an epimorphism.

Conversely, if $j_\Delta$ is an epimorphism then so are 
its components  $j_{X*}$ and $j_{Y*}$.
\end{proof}

The term ``minimal" is unsatisfactory for several reasons, and we shall henceforth say that
an embedding satisfying the equivalent conditions of Lemma \ref{minimal} is {\it bi-epic}.
Embeddings obtained from other embeddings by nontrivial ``2-knot surgery" \cite{Hi17}
are never bi-epic.
However, if $j=j_L$ for some bipartedly ribbon link $L$ then $j$ is bi-epic, since $\pi$, 
$\pi_X$ and $\pi_Y$ are generated by images of the meridians of $L$.

\begin{ex}
There are $3$-manifolds with more than one bi-epic embedding.
\end{ex}

The link $L$ obtained from the Borromean rings by replacing 
one component by its $(2,1)$-cable and another by its $(3,1)$-cable 
may be partitioned as the union of two trivial links in three ways. 
The resulting three embeddings of $M(L)$ in $S^4$ each have $Y\simeq{S^1}\vee2S^2$,
but the groups $\pi_X$ have presentations $\langle{a,b}|[a,b^2]^3\rangle$,
$\langle{a,c}|[a,c^3]^2\rangle$, and $\langle{b,c}|[b^2,c^3]\rangle$, 
respectively, and so are distinct.
In the first two cases $\pi$ has torsion, 
while in the third case $X$ is aspherical.
(None of these groups is  abelian.)
This example can obviously be generalized in various ways.
The homology sphere in  \cite[Figure 3]{Hi17} is another example;
the embedding determined by the link is bi-epic, but the 3-manifold also has
an embedding with both complementary regions contractible.
However the latter embedding may not derive from a
0-framed link representing the homology sphere.

The cases when $j_\Delta$ is an isomorphism are quite rare.

\begin{lemma}
If $j_\Delta$ is an isomorphism then either $M\cong{F}\times{S^1}$ for some 
aspherical closed orientable surface $F$ or $M\cong\#^r(S^2\times{S^1})$ for some $r\geq0$.
\end{lemma}

\begin{proof}
If $\pi\cong\pi_X\times\pi_Y$ with $\pi_X$ infinite and $\pi_Y\not=1$ then
$M\cong{F}\times{S^1}$ for some aspherical closed orientable surface $F$ \cite{Ep}.
If $\pi_Y=1$ then $j_{X*}$ is an isomorphism, and so $\pi$ must be a free group \cite{Da}.
Hence $M\cong\#^r(S^2\times{S^1})$ for some $r\geq1$.
Finally, if $\pi_X$ and $\pi_Y$ are both finite and have nontrivial abelianization
then their orders have a common prime factor $p$, and so $\pi$ has
$(\mathbb{Z}/p\mathbb{Z})^2$ as a subgroup, which is not possible.
We may also exclude $\pi_X\cong\pi_Y\cong{I^*}$, for a similar reason,
and so there remains only the case $\pi=1$, when $M=S^3=\#^0(S^2\times{S^1})$.
\end{proof}

These 3-manifolds do in fact have bi-epic embeddings with $j_\Delta$ an isomorphism.

\section{homotopy equivalences}

In this section we shall give some lemmas on recognizing 
the homotopy types of certain spaces and pairs of spaces arising later.
One simple but important observation is that the natural homomorphisms
$H_2(X;\mathbb{Z})\to{H_2(X,M;\mathbb{Z})}$ is 0, since it factors through 
$H_2(S^4;\mathbb{Z})\to{H_2(S_4,Y;\mathbb{Z})}$,
and similarly for $H_2(Y;\mathbb{Z})\to{H_2(Y,M;\mathbb{Z})}$.
Equivalently, the intersection pairings are trivial on $H_2(X;\mathbb{Z})$ and $H_2(Y;\mathbb{Z})$.
(See Theorem \ref{homhandle} below for one use of this observation.)

\begin{theorem}
\label{2con}
Let $U$ and $V$ be connected finite cell complexes such that $c.d.U\leq2$ and $c.d.V\leq2$.
If $f:U\to{V}$ is a $2$-connected map then $\chi(U)\geq\chi(V)$,
with equality if and only if $f$ is a homotopy equivalence.
\end{theorem}

\begin{proof}
Up to homotopy, we may assume that $f$ is a cellular inclusion,
and that $V$ has dimension $\leq3$.
Let $\pi=\pi_1(U)$ and let $C_*=C_*(\widetilde{V},\widetilde{U})$. 
Then $H_q(C_*)=0$ if $q\leq2$, since $f$ is 2-connected,
and $H_q(C_*)=0$ if $q>3$, since $c.d.U$ and $c.d.V\leq2$.
Hence $H_3(C_*)\oplus{C_2}\oplus{C_0}\cong{C_3}\oplus{C_1}$,
by Schanuel's Lemma,
and so $H_3(C_*)$ is a stably free $\mathbb{Z}[\pi]$-module 
of rank $-\chi(C_*)=\chi(U)-\chi(V)$.
Hence $\chi(U)\geq\chi(V)$, with equality if and only if
$H_3(C_*)=0$, since group rings are weakly finite,
by a theorem of Kaplansky.
(See \cite{Ro} for a proof.)
The result follows from the long exact sequence of the pair
$(\widetilde{Y},\widetilde{X})$ and the theorems of Hurewicz and Whitehead.
\end{proof}

If $c.d.X\leq2$ then $C_*(\widetilde{X})$ is chain homotopy equivalent 
to a finite projective complex of length 2,
which is a partial resolution of the augmentation module $\mathbb{Z}$.
Chain homotopy classes of such partial resolutions 
are classified by $Ext^3_{\mathbb{Z}[\pi]}(\mathbb{Z},\Pi)=H^3(\pi;\Pi)$,
where $\Pi$ is the module of 2-cycles.

\begin{cor}
\label{2concor}
If $U$ is a connected finite complex such that $c.d.U\leq2$ and $\pi_1(U)\cong\mathbb{Z}$
then $U\simeq{S^1}\vee\bigvee^{\chi(U)}S^2$.
\end{cor}

\begin{proof}
Since $c.d.U\leq2$ and projective $\mathbb{Z}[\pi_1(U)]$-modules are free,
$C_*(\widetilde{U})$ is chain homotopy equivalent to 
a finite free $\mathbb{Z}[\pi_1(U)]$-complex $P_*$ of length $\leq2$,
and $\chi(U)=\Sigma(-1)^irank(P_i)$.
Since $\pi_2(U)\cong{H_2(U;\mathbb{Z}[\pi_1(U)])}$
is the module of 2-cycles in  $C_*(\widetilde{U})$, it is free of rank $\chi(U)$.
Let $f:{S^1\vee\bigvee^{\chi(U)}S^2\to{U}}$ be the map determined by a
generator for $\pi_1(U)$ and representatives of a basis for $\pi_2(U)$.
Then $f$ is a homotopy equivalence, by the theorem.
\end{proof}

Theorem 3.2 of \cite{Hi} gives an analogue of Theorem \ref{2con} for maps between closed 4-manifolds.
The argument extends to the following relative version.

\begin{lemma}
\label{4man}
Let $f:(X_1,A_1)\to(X_2,A_2)$ be a map of orientable $PD_4$-pairs such that
$f|_{A_1}:A_1\to{A_2}$ is a homotopy equivalence. Then $f$ is a homotopy
equivalence of pairs if and only if $\pi_1(f)$ is an isomorphism and $\chi(X_1)=\chi(X_2)$.
\end{lemma}

\begin{proof}
Since $f|_{A_1}:A_1\to{A_2}$ is a homotopy equivalence,
$f$ has degree 1, and hence is 2-connected as a map from $X_1$ to $X_2$.
The rest  of the argument is as in \cite[Theorem 2]{Hi}.
\end{proof} 

In certain cases we can identify the homotopy type of a pair.

\begin{lemma}
\label{asph}
Let $(X,A)$ and $(X',A')$ be pairs such that the inclusions 
${\iota_A:A\to{X}}$ and $\iota_{A'}:A'\to{X'}$ induce epimorphisms 
on fundamental groups.
If $X$ and $X'$ are aspherical and $f:A\to{A'}$ 
is a homotopy equivalence such that
$\pi_1(f)(\mathrm{Ker}(\pi_1(\iota_A)))=\mathrm{Ker}(\pi_1(\iota_{A'}))$
then $f$ extends to a homotopy equivalence of pairs $(X,A)\simeq(X',A')$.
\end{lemma}

\begin{proof}
The fundamental group conditions imply that $g=\iota_{A'}f$ 
extends to a map from the relative 2-skeleton $X^{[2]}\cup{A}$.
The further obstructions to extending $g$ to a map from $X$ to $X'$ lie in $H^{q+1}(X,A;\pi_q(X'))$,
for $q\geq2$. Since $X'$ is aspherical these groups are 0.
The other hypotheses imply that any extension $h:X\to{X'}$ induces 
an isomorphism on fundamental groups,
and hence is a homotopy equivalence.
\end{proof}

We would like to have an analogue of Lemma \ref{asph} for
the cases when $\pi_X\cong\mathbb{Z}$ and $\chi(X)=1$.
If $(X,\partial{X})$ is a $PD_4$-pair such that $X\simeq{S^1}\vee{S^2}$
then $\pi_2(X)\cong\mathbb{Z}[\pi_X]$ and $\pi_3(X)\cong\Gamma_W(\mathbb{Z}[\pi_X])$,
where  $\Gamma_W$ is the quadratic functor of Whitehead.
Let $(X,\partial{X})$ and  $(\widehat{X},\partial{\widehat{X}})$ 
be  two such $PD_4$-pairs, and let $\iota_X$ and $\iota_{\widehat{X}}$
be the inclusions of the boundaries. 
Then any homotopy equivalence $f:\partial{X}\to\partial{\widehat{X}}$ 
such that $f\iota_X\sim\iota_{\widehat{X}}$
extends across the relative 3-skeleton $X^{[3]}\cup\partial{X}$, since
$H^3(X,\partial{X};f^*\pi_2(\widehat{X}))\cong{H_1(X;\mathbb{Z}[\pi_X])}=0$.
The only obstruction to extending such an $f$ to a map from $X$ to $\widehat{X}$
lies in $H^4(X,\partial{X};f^*\pi_3(\widehat{X}))\cong{H_0(X;f^*\pi_3(\widehat{X}))}
\cong\mathbb{Z}\otimes_{\mathbb{Z}[\pi_X]}\Gamma_W(\mathbb{Z}[\pi_X])$.
(Any such extension would be a homotopy equivalence.)
This obstruction is perhaps determined by the equivariant intersection pairings on
$\pi_2(X)$ and $\pi_2(\widehat{X})$.
Can we use the additional constraints that $(X,\partial{X})$ and  $(\widehat{X},\partial{\widehat{X}})$ 
are codimension-0 submanifolds of $S^4$?
(Note also that a further extension to the case when $\pi_1(Y)\cong\mathbb{Z}$ and $\chi(Y)>1$
would imply the Unknotting Theorem for orientable surfaces in $S^4$.)

\section{abelian embeddings}

In so far as we hope to apply 4-dimensional topological surgery  
to the complementary regions, 
we need to assume that $\pi_X$ and $\pi_Y$ are ``good" 
in the sense of \cite{FQ}.
At present, the class of groups known to be good is somewhat larger 
than the class of elementary amenable groups.

The subclass of nilpotent groups is of particular interest.  
If $\pi_X$ is nilpotent then $j_{X*}$ is onto, 
since $H_1(j_X)$ is onto,
and any subset of a nilpotent group $G$ whose image generates
the abelianization $G/G'$ generates $G$.
Since $j_{X*}$ is onto, $c.d.X\leq2$, by \cite[Theorem 5.1]{Hi17}.
(Similarly, if $\pi_Y$ is nilpotent then $j_{Y*}$ is onto and $c.d.Y\leq2$.)
If $\pi_X$ and $\pi_Y$ are each nilpotent then $j$ is bi-epic,
by Lemma \ref{minimal} above.
There are also purely algebraic reasons why nilpotent groups 
should be of particular interest. 
Firstly, there is the well-known connection between homology,
lower central series and (Massey) products (as used in \cite{Hi17}).
Secondly, if a group $G$ is finite or solvable and every homomorphism
$f:H\to{G}$ which  induces an epimorphism on abelianization is an epimorphism
then $G$ must be nilpotent. 
(See pages 132 and 460 of \cite{Rob}.)

The consequences of the Mayer-Vietoris sequence noted above together with the fact that
the higher $L^2$ Betti numbers of amenable groups vanish
give a simple but useful constraint.

\begin{lemma}
\label{L2}
If $c.d.X\leq2$ and $\beta^{(2)}_1(\pi_X)=0$ then either $\chi(X)=0$ and $X$ is aspherical 
or $\chi(X)=1$.
\end{lemma}

\begin{proof}
This follows from a mild extension of  \cite[Theorem 2.5]{Hi}.
Since $c.d.X\leq2$ and $X$ is homotopy equivalent to a finite 3-complex,
$C_*(\widetilde{X})$ is chain homotopy equivalent to 
a finite free $\mathbb{Z}[\pi_X]$-complex $D_*$ of length at most 2.
If $\beta_1^{(2)}(\pi_X)=0$ then $\chi(X)=\chi(D_*)=dim_{\mathcal{N}(\pi_X)}H_2(D_*)\geq0$,
with equality only if $D_*$ is acyclic,
in which case  $X$ is aspherical.
\end{proof}

In particular, if $\pi_X$ is elementary amenable and $\chi(X)=0$
then $\pi_X\cong\mathbb{Z}$ or $\mathbb{Z}*_m$ 
(with presentation $\langle{a,t}|tat^{-1}=a^m\rangle$), for some $m\not=0$.
(See \cite[Corollary 2.6.1]{Hi}.)

The first $L^2$-Betti number vanishes also for semidirect products $N\rtimes\mathbb{Z}$
with $N$ finitely generated. 
(This observation is used in Theorem \ref{mtor} below.)

\begin{ex}
If $M=M(-2;(1,0))$ or $M(-2;(1,4))$ and $j$ is bi-epic then $X\simeq{Kb}$.
\end{ex}

In each case $\pi$ is polycyclic and $\pi/\pi'\cong\mathbb{Z}\oplus(\mathbb{Z}/2\mathbb{Z})^2$.
Hence $\chi(X)=0$, and so $c.d.\pi_X\leq2$.
Since $\pi_X$ is a quotient of $\pi$ and $\pi_X/\pi_X'\cong\mathbb{Z}\oplus\mathbb{Z}/2\mathbb{Z}$
we must have $\pi_X\cong\mathbb{Z}*_{-1}=\mathbb{Z}\rtimes_{-1}\!\mathbb{Z}$.
Since $c.d.X\leq2$ and $\chi(X)=0$ the classifying map $c_X:X\to{Kb}=K(\mathbb{Z}\rtimes_{-1}\!\mathbb{Z},1)$
is a homotopy equivalence.

\smallskip
If we restrict further to the abelian case the possible groups are known.
If $\pi_X$ is abelian and $\chi(X)=0$ then either 
$\pi_X\cong\mathbb{Z}$ and $H_2(X;\mathbb{Z})=0$ 
or $\pi_X\cong\mathbb{Z}^2$ and $H_2(X;\mathbb{Z})\cong\mathbb{Z}$.
If $\pi_X$ is abelian and $\chi(X)=1$ then $\chi(Y)=1$ also,
and so $\beta_1(\pi_X)\geq\beta_2(\pi_X)$.
In the latter case it follows easily that 
$\pi_X\cong\mathbb{Z}/k\mathbb{Z}$, $\mathbb{Z}\oplus\mathbb{Z}/k\mathbb{Z}$,
$\mathbb{Z}^2$ or $\mathbb{Z}^3$.
Hence either 
$\beta=0$ and $\pi_X\cong\mathbb{Z}/k\mathbb{Z}$
or $\beta=2$ and $\pi_X\cong\mathbb{Z}\oplus\mathbb{Z}/k\mathbb{Z}$,
for some $k\geq1$, or $\beta=1,3,4$ or 6 and 
$\pi_X\cong\mathbb{Z}^{\lfloor\!\frac{\beta+1}2\!\rfloor }$.
(See also \cite[Theorem 7.1]{Hi17}.)

\begin{lemma}
\label{finite}
If  $\pi_X$ is abelian of rank at most $1$ 
then $X$ is homotopy equivalent to a finite $2$-complex.
If moreover $\pi_X$ is cyclic then $X\simeq{S^1}$ or $S^1\vee_\ell{S^2}$,
for some $\ell\in\mathbb{Z}$.
or $P_\ell=S^1\vee_\ell{e^2}$, with $\ell\not=0$.
\end{lemma}

\begin{proof}
The first assertion follows from the facts that $c.d.X\leq2$,
as just observed, 
and that the $\mathcal{D}(2)$ property holds for
cyclic groups (see \cite[page 235]{Jo}) 
and for the groups $\mathbb{Z}\oplus\mathbb{Z}/\ell\mathbb{Z}$ \cite{Ed}. 
If $\pi_X\cong\mathbb{Z}/\ell\mathbb{Z}$ is cyclic then  $X\simeq{S^1}$ or $S^1\vee{S^2}$,
if $\ell=0$, or $P_\ell=S^1\vee_\ell{e^2}$, if $\ell\not=0$ \cite{DS}.
\end{proof}

We shall show later that a similar result holds when $\pi_X\cong\mathbb{Z}^2$.

We shall say that an embedding $j$ is {\it abelian\/} or {\it nilpotent\/} 
if $\pi_X$ and $\pi_Y$ are each abelian or nilpotent, respectively.
Ten of the thirteen 3-manifolds with elementary amenable 
fundamental groups and which embed in $S^4$ 
(see \cite{CH}) have abelian embeddings.
(Apart from the Poincar\'e homology 3-sphere $S^3/I^*$, 
which does not embed smoothly,
these derive from the empty link $\emptyset$, 
the unknot $U$, the 2-component links $4^2_1$, $5^2_1$, $6^2_1$, $8^2_2$,
$9^2_{53}$, $9^2_{61}$ and  the Borromean rings $6^3_2$.)
In at least four cases ($S^3$, $S^3/Q(8)$, $S^3/I^*$ and $S^2\times{S^1}$)
the abelian embedding is essentially unique.
The ``half-turn" flat 3-manifold $G_2=M(-2;(1,0))$ 
and the $\mathbb{N}il^3$-manifold $M(-2;(1,4))$ 
bound regular neighbourhoods of embeddings 
of the Klein bottle $Kb$ in $S^4$,
but have no abelian embeddings.
The status of one $\mathbb{S}ol^3$-manifold is not yet known.

When $\beta\leq1$ we must have $\chi(X)=1-\beta$.
If $L$ is a 2-component slice link with unknotted components 
(such as the trivial 2-component link, or the Milnor boundary link)
and $M=M(L)$ then $\beta=2$ and $M$ has an abelian embedding (with $\chi(X)=\chi(Y)=1$), 
and also an embedding with $\chi(X)=1-\beta=-1$ and $\pi_Y=1$.
However it shall follow from the next lemma that if $\beta>2$ then $M$ 
cannot have both an abelian embedding and also one with $\chi(X)=1-\beta$.

\begin{lemma}
 \label{completion}
Let $j:M\to{S^4}$ be an embedding $j$ such that $H_1(Y;\mathbb{Z})=0$,
and let $S\subset\Lambda_\beta=\mathbb{Z}[\pi/\pi']$ be the multiplicative system consisting 
of all elements $s$ with augmentation $\varepsilon(s)=1$.
If the augmentation homomorphism $\varepsilon:\Lambda_{\beta{S}}\to\mathbb{Z}$
factors through an integral domain $R\not=\mathbb{Z}$ then $H_1(M;R)$ has rank $\beta-1$ as an $R$-module.
\end{lemma}

\begin{proof}
Let  $*$ be a basepoint for $M$ and $A(\pi)=H_1(M,*;\Lambda_\beta)$ 
be the Alexander module of $\pi$.
(See  \cite[Chapter 4]{HiA}.)
Since $H_2(X;\mathbb{Z})=0$,
the inclusion of representatives for a basis of $H_1(X;\mathbb{Z})\cong\mathbb{Z}^\beta$
induces isomorphisms $F(\beta)/F(\beta)_{[n]}\cong\pi/\pi_{[n]}$, for all $n\geq1$,
by a theorem of Stallings. (See \cite[Lemma 3.1]{Hi17}.)
Hence $A(\pi)_S\cong(\Lambda_{\beta{S}})^\beta$,
by \cite[Lemma 4.9]{HiA}.
Since $\varepsilon$ factors through $R$, 
the exact sequence of the pair $(M,*)$
with coefficients $R$ gives an exact sequence
\[
0\to{H_1(M;R)}\to{R}\otimes_{\Lambda_\beta}{A(\pi)}\cong{R^\beta}\to{R}\to{R}\otimes_{\Lambda_\beta}\mathbb{Z}=
\mathbb{Z}\to0,
\]
from which the lemma follows.
(Note that the hypotheses on $R$ imply that $\mathbb{Z}$ is an $R$-torsion module.)
\end{proof}

\section{homology spheres and handles}

If $M$ is an integral homology 3-sphere then it bounds a contractible 4-manifold,
and so has an abelian embedding with $X$ and $Y$ each contractible. 
They are determined up to homeomorphism by their boundaries \cite{FQ}, 
and so the abelian embedding is unique.
Moreover, the complementary regions are homeomorphic.
When $M=S^3$, the result goes back to the Brown-Mazur-Schoenflies Theorem,
which does not use surgery. 
(In this special case the embedding is essentially unique!)

It is not clear whether non-simply connected homology spheres must have
embeddings with one or both of $\pi_X$ and $\pi_Y$ nontrivial. 
Figure 3 of \cite{Hi17} gives an example with $\pi_X\cong\pi_Y\cong{I^*}$,
the binary icosahedral group.
In this case the homology sphere is the result of surgery on a complicated
4-component bipartedly trivial link, and probably has no simpler description.
The Poincar\'e homology 3-sphere $S^3/I^*$ is not the result 
of 0-framed surgery on any bipartedly slice link,
since it does not embed smoothly.

If instead $M$ is an orientable homology handle,  i.e., 
if $\pi/\pi'=H_1(M;\mathbb{Z})\cong\mathbb{Z}$, 
so $M$ has the homology of $S^2\times{S^1}$,
then $\pi'/\pi''$ is a finitely generated torsion module over 
$\mathbb{Z}[\pi/\pi']\cong\Lambda=\mathbb{Z}[t,t^{-1}]$.
Equivariant Poincar\'e duality and the universal coefficient theorem 
together define a nonsingular hermitean pairing $b$ on $\pi'/\pi''$, 
with values in $\mathbb{Q}(t)/\Lambda$, called the Blanchfield pairing.
The pairing is {\it neutral\/} if $\pi'/\pi''$ has a submodule $N$ 
which is its own annihilator with respect to $b$, i.e., 
such that $N=\{m\in\pi'/\pi''\mid{b(m,n)=0}~\forall{n}\in{N}\}$.
A knot $K$ is {\it algebraically slice\/} if the Blanchfield pairing of $M(K)$
is neutral.

\begin{theorem}
\label{neutral}
Let $M$ be  an orientable homology handle.
If $M$ embeds in $S^4$ then the Blanchfield pairing on 
$\pi'/\pi''=H_1(M;\mathbb{Z}[\pi/\pi'])$ is neutral.
There is an abelian embedding $j:M\to{S^4}$
if and only if $\pi'$ is perfect,
and then $X\simeq{S^1}$ and $Y\simeq{S^2}$.
\end{theorem}

\begin{proof}
The first assertion follows on applying equivariant Poincar\'e-Lefshetz 
duality to the infinite cyclic cover of the pair $(X,M)$.
(See the proof of  \cite[Theorem 2.4]{HiA}.)

If $j$ is abelian then $\pi_X\cong\mathbb{Z}$ and $\pi_Y=1$,
while $H_2(X;\mathbb{Z})=0$  and $H_2(Y;\mathbb{Z})\cong\mathbb{Z}$.
Since $c.d.X\leq2$ and $\pi_X\cong\mathbb{Z}$,
it follows that $\pi_2(X)=H_2(X;\mathbb{Z}[\pi_X])$ 
is a free $\mathbb{Z}[\pi_X]$-module of rank $\chi(X)=0$.
Hence $\pi_2(X)=0$, and so maps $f:S^1\to{X}$ and $g:S^2\to{Y}$ 
representing generators for $\pi_X$ and $\pi_2(Y)$ are homotopy equivalences.
Since $H_2(X,M;\mathbb{Z}[\pi_X])\cong\overline{H^2(X;\mathbb{Z}[\pi_X])}=0$, 
by equivariant Poincar\'e-Lefshetz duality,
$\pi'/\pi''=H_1(M;\mathbb{Z}[\pi_X])=0$, by the homology exact sequence 
for the infinite cyclic cover of the pair $(X,M)$.
Hence $\pi'$ is perfect.

Suppose, conversely, that $\pi'$ is perfect.
Then $M$ embeds in $S^4$, by the main result of \cite{Hi96},
and examination of the proof shows that the embedding constructed 
in the theorem is abelian.
\end{proof}

Since $S^2\times{S^1}$  may be obtained by 0-framed 
surgery on the unknot, it has a standard  abelian embedding
with $X\cong{S^1\times{D^3}}$ and $Y\cong{D^2}\times{S^2}$.
(In fact $Y\cong{D^2}\times{S^2}$ whenever $M=S^2\times{S^1}$,
by a result of Aitchison \cite{Ru}.)

If $K$ is an Alexander polynomial 1 knot then 
$M(K)$ has an abelian embedding, 
and if $K$ is a knot such that $M(K)$ embeds in $S^4$ then
$K$ is algebraically slice, by Theorem \ref{neutral}.
However if $K$ is a slice knot with nontrivial Alexander polynomial 
then $M(K)$ embeds in $S^4$ but no embedding is abelian.
There are obstructions beyond neutrality of the Blanchfield pairing
to slicing a knot, 
which probably also obstruct embeddings of homology handles.

\begin{theorem}
\label{homhandle}
Let $M$ be an orientable homology handle.
Then $M$ has at most one abelian embedding, up to equivalence.
\end{theorem}

\begin{proof}
Assume that $j_1$ and $j_2$ are abelian embeddings of $M$.
There is a homotopy equivalence of pairs 
$(X_1,M)\simeq(X_2,M)$ which extends $id_M$, 
by Lemma \ref{asph}.
This is homotopic {\it rel} $M$ to a  homeomorphism $F$,
since the surgery obstruction group $L_5(\mathbb{Z})$ acts trivially on 
the structure set $\mathcal{S}_{TOP}(X_2,\partial{X_2})$.
(This follows from the Wall-Shaneson theorem and the existence of the
$E_8$-manifold \cite[Theorem 6.7] {Hi}.)

We may assume the homotopy equivalences $Y_1\simeq{S^2}$ and $Y_2\simeq{S^2}$
are so chosen that the corresponding maps  $f_1$ and $f_2$ from $M$ to $S^2$
induce the same class in $H^2(M;\mathbb{Z})$.
We may also assume that $f_1$ and $f_2$ agree on the 2-skeleton of $M$,
by  \cite[Theorem 8.4.11]{Sp}.
Let $p:M\to{M}\vee{S^3}$ be a pinch map, 
and $\eta:S^3\to{S^2}$ be the Hopf fibration.
Let $d_t$ be a self map of $S^3$ of degree $t$,
and let $q_t=id_M\vee{d_t}$.
Then $f_2\sim(f_1\vee\eta)q_tp$, 
for some $t\in\mathbb{Z}$. 
Let $Z_i$ be the mapping cylinder of $f_i$, for $i=1,2$.
Then $Y_i$ is homotopy equivalent to $Z_i$ {\it rel} $M$, for $i=1,2$. 

Let $P=MCyl(\eta)=\overline{\mathbb{CP}^2\setminus{D^4}}$
and $W=MCyl(f_1\vee\eta)$. 
The inclusions of $S^3$ and $M$ into $M\vee{S^3}$ and $q_tp$
induce maps $\theta$, $\psi$ and $\xi$ from $(P,S^3)$,
$(Z_1,M)$ and $(Z_2,M)$, respectively, to $(W,M\vee{S^3})$.
These induce isomorphisms of $H^2(P;\mathbb{Z})$, 
$H^2(Z_1;\mathbb{Z})$ and $H^2(Z_2;\mathbb{Z})$ 
with $H^2(W;\mathbb{Z})\cong\mathbb{Z}$.
Let $\sigma$ generate $H^2(W;\mathbb{Z})$.
The groups $H_4(P,S^3;\mathbb{Z})$, 
$H_4(Z_1,M;\mathbb{Z})$ and $H_4(Z_2,M;\mathbb{Z})$ 
are also infinite cyclic,
with generators $[P,S^3]$, $[Z_1,M]$ and $[Z_2,M]$, respectively,
but $H_4(W,M\vee{S^3};\mathbb{Z})\cong
{H_3(S^3;\mathbb{Z})}\oplus{H_3(M;\mathbb{Z})}$,
and $\xi_*[Z_2,M]=t.\theta_*[P,S^3]+\psi_*[Z_1,M]$.
Hence
\[
\xi^*\sigma^2\cap[Z_2,M]=\sigma^2\cap\xi_*[Z_2,M]=
t\sigma^2\cap\theta_*[P,S^3]+\sigma^2\cap\psi_*[Z_1,M].
\]
The inclusion of $(P,S^3)$ into $(\mathbb{CP}^2,D^4)$ 
induces isomomorphisms on (relative) cohomology, 
and so 
$\sigma^2\cap\theta_*[P,S^3]=\theta^*\sigma^2\cap[P,S^3]\not=0$.
Since the middle dimensional intersection pairing 
is trivial in each of $(Z_1,M)$ and $(Z_2,M)$, $t=0$, 
and so $f_1\sim{f_2}$.
Hence there is a homotopy equivalence of pairs $(Y_1,M)\to(Y_2,M)$
which extends $id_M$.
This is homotopic {\it rel} $M$ to a homeomorphism $G$, 
by simply-connected surgery.
The map $h=F\cup{G}$ is a homeomorphism of $S^4$ such that $hj_1=j_2$. 
\end{proof}

Part of the argument for Theorem \ref{homhandle} was suggested by Section 2 of \cite{La}.

\begin{ex}
The manifold $M=M(11_{n42})$ has an essentially unique abelian embedding,
although $M=M(K)$ for infinitely many distinct knots $K$.
\end{ex}

The knot $11_{n42}$ is the Kinoshita-Terasaka knot,
which is the simplest non-trivial knot with Alexander polynomial 1.
This bounds a smoothly embedded disc $D$ in $D^4$,
such that $\pi_1(D^4\setminus{D})\cong\mathbb{Z}$,
obtained by desingularizing a ribbon disc.
(See Figure 1.4 of \cite{HiA}.)
Hence $M$ has a smooth abelian embedding.
Since $11_{n42}$ has unkotting number 1, 
it has an annulus presentation, 
and so there are infinitely many knots $K_n$ such that $M(K_n)\cong{M}$ \cite{AJOT}. 
These knots must all have Alexander polynomial 1,
and so each determines an abelian embedding.
Are all of these embeddings smooth, and are they smoothly equivalent?

The connected sum of the homology 3-sphere represented by Figure 3 
of \cite{Hi17} with $S^2\times{S^1}$ has an embedding with 
both complementary regions having nontrivial fundamental group.
Does every nontrivial homology handle have an embedding for which $\pi_Y\not=1$?

A 1-knot $K$ is {\it homotopically ribbon\/} 
if it bounds a disc $D\subset{D^4}$ such that the inclusion of
$M(K)=\partial{D^4\setminus{D}}$ into $D^4\setminus{D}$
induces an epimorphism on fundamental groups \cite{CG}.

\begin{theorem}
Let $K$ be a $1$-knot. 
Then $M=M(K)$ has a bi-epic embedding if and only if
$K$ is homotopy ribbon.
\end{theorem}

\begin{proof}
If $K$ is homotopically ribbon then the embedding corresponding to
the slice disc demonstrating this property is clearly bi-epic.

Suppose that $M$ has a bi-epic embedding.
Let $W$ be the trace of 0-framed surgery on $K$.
Then $W$ is 1-connected, $\chi(W)=1$ and $\partial{W}=S^3\amalg{M}$.
Let $P=X\cup_MW$. 
Then $P$ is 1-connected, since $\pi(j_X)$ is an epimorphism,
$\chi(P)=1$, and $\partial{P}=S^3$, and so $P\cong{D^4}$.
Clearly $K$ is homotopy ribbon in $P$.
\end{proof}

For example, 
if $k$ is a fibred 1-knot with exterior $E(k)$ and genus $g$, 
then $K=k\#\!-\!k$ is a fibred ribbon knot,
and $M(K)$ bounds a thickening $X$ of $E(k)\subset{S^3}\subset{S^4}$,
which fibres over $S^1$, with fibre $\natural^g(S^1\times{D^2})$.

If an orientable homology handle $M$ is Seifert fibred  
then the base orbifold is orientable,
since $M$ is orientable and $\pi/\pi'$ is torsion free.
Hence $M$ has generalized Euler invariant $\varepsilon=0$, 
since $\beta$ is odd, and so is a mapping torus. 
It then follows from Theorem \ref{neutral} that if  $M$ 
has an abelian embedding then $M\cong{S^2\times{S^1}}$.

In the next theorem we do not assume here that $M$ is an homology handle.

\begin{theorem}
\label{mtor}
Let $M$ be a $3$-manifold.
If  $j:M\to{S^4}$ is an embedding such that $X$ fibres over $S^1$ then  $\chi(X)=0$,
$M$ is a mapping torus, the projection $p:M\to{S^1}$ extends to a map
from $X$ to $S^1$ and $\pi_1(j_X)$ is surjective.
Conversely, if these conditions hold then there is an embedding 
$\widehat{j}:M\to{S^4}$ 
such that $\widehat{X}$ fibres over $S^1$ and $(\widehat{X},M)$ 
is $s$-cobordant rel $M$ to $(X,M)$. 
\end{theorem}

\begin{proof}
If $X$ fibres over $S^1$, with fibre $F$, 
then $M=\partial{X}$ is the mapping torus
of a self-homeomorphism of $\partial{F}$ and the projection $p:M\to{S^1}$ 
extends to a map from $X$ to $S^1$.
Moreover, $\chi(X)=0$ and $\pi_X$ is an extension of $\mathbb{Z}$
by the finitely presentable normal subgroup $\pi_1(F)$.
Hence $\beta_1^{(2)}(\pi_X)=0$,
by  \cite[Theorem 7.2.6]{Lue},
and so $c.d.\pi_X\leq2$, by Lemma \ref{L2}.
Hence $\pi_1(F)$ is free, by  \cite[Corollary 6.6]{Bi},
and so $F\cong\#^r(S^1\times{D^2})$, for some $r\geq0$.
Moreover,  $\pi_1(j_X)$ is surjective.

If $M$ is a mapping torus, the projection $p:M\to{S^1}$ extends to a map 
from $X$ to $S^1$ and $\pi_1(j_X)$ is surjective then $\pi_X$ 
is an extension of $\mathbb{Z}$ by a finitely presentable normal subgroup.
Since $\chi(X)=0$, the space $X$ is aspherical, 
and so $\pi_X\cong{F(r)\rtimes\mathbb{Z}}$, for some $r\geq0$.
Let $X^\infty$ be the covering space associated to the subgroup $F(r)$, 
and  let $j_{X^\infty}$ be the inclusion of $M^\infty=\partial{X^\infty}$ into $X^\infty$.
Let $\tau$ be a generator of the covering group $\mathbb{Z}$.
Fix a homotopy equivalence $h:X^\infty\to{N}=\#^r(S^1\times{D^2})$.
Then there is a self-homeomorphism $t_N$ of $N$ such that $t_Nh\sim{h}\tau$.
Let $\theta:\partial{N}\to{N}$ be the inclusion,
and let $\widehat{X}=M(t_N)$ be the mapping torus of $t_N$.
Then there is a homotopy equivalence $\alpha:M^\infty\to\partial{N}$ 
such that $\theta\alpha\sim{hj_{X^\infty}}$,
by a result of Stallings and Zieschang. (See \cite[Theorem 2]{GK}.)
We may modify $h$ on a collar neighbourhood of $\partial{X^\infty}$ so that 
$h|_{\partial{X^\infty}}=\alpha$.
Hence $h$ determines a homotopy equivalence of pairs 
$(X,M)\simeq(\widehat{X},\partial\widehat{X})$.
Since $M$ and $\partial\widehat{X}$ are orientable (Haken) manifolds
we may further arrange that $h|_M:M\to\partial\widehat{X}$ is a homeomorphism.
Hence $X$ and $\widehat{X}$ are $s$-cobordant {\it rel} $\partial$,
since $L_5(F(r))$ acts trivially on the $s$-cobordism structure set 
$\mathcal{S}^s_{TOP}(\widehat{X},\partial\widehat{X})$.
(See \cite[Theorem 6.7]{Hi}.)

The union $\Sigma=\widehat{X}\cup_MY$ is an homotopy 4-sphere, and so is homeomorphic to $S^4$.
Then the final assertion is satisfied by the composite $\widehat{j}:M\subset\widehat{X}\subset\Sigma\cong{S^4}$.
\end{proof}

In particular, if $\beta=1$ then $\chi(X)=0$ and $M$ is a rational homology handle.

\section{$\pi/\pi'\cong\mathbb{Z}^2$}

When  $\pi/\pi'\cong\mathbb{Z}^2$ there is again a simple necessary 
condition for $M$ to have an abelian embedding.

\begin{theorem}
\label{Z^2}
Let $M$ be a $3$-manifold with fundamental group $\pi$ 
such that $\pi/\pi'\cong\mathbb{Z}^2$.
If $j:M\to{S^4}$ is an abelian embedding then 
$X\simeq{Y}\simeq{S^1\vee{S^2}}$, and
$H_1(M;\mathbb{Z}[\pi_X])$ and $H_1(M;\mathbb{Z}[\pi_Y])$ 
are cyclic $\mathbb{Z}[\pi_X]$- and $\mathbb{Z}[\pi_Y]$-modules
(respectively), of projective dimension $\leq1$.
\end{theorem}

\begin{proof}
Since $j$ is abelian $\pi_X\cong\pi_Y\cong\mathbb{Z}$ and $\chi(X)=\chi(Y)=1$.
Moreover,
since $j_{X*}$ and $j_{Y*}$ are epimorphisms $c.d.X\leq2$ and $c.d.Y\leq2$,
by  \cite[Theorem 5.1]{Hi17}.
Hence $X\simeq{Y}\simeq{S^1\vee{S^2}}$, by Corollary \ref{2concor}.

As in Theorem \ref{neutral} we consider the homology exact sequences
of the infinite cyclic covers of the pairs $(X,M)$ and $(Y,M)$,
in conjunction with equivariant Poincar\'e-Lefshetz duality.
Since $H_i(X;\mathbb{Z}[\pi_X])=0$  for $i\not=0$ or 2
and $H_2(X;\mathbb{Z}[\pi_X])\cong\mathbb{Z}[\pi_X]$,
we have $H_2(X,M;\mathbb{Z}[\pi_X])\cong
\overline{H^2(X;\mathbb{Z}[\pi_X])}\cong\mathbb{Z}[\pi_X]$ also.
Hence there is an exact sequence 
\[
0\to{H_2(M;\mathbb{Z}[\pi_X])}\to\mathbb{Z}[\pi_X]\to
\mathbb{Z}[\pi_X]\to{H_1(M;\mathbb{Z}[\pi_X])}\to0.
\]
Therefore either $H_1(M;\mathbb{Z}[\pi_X])\cong{H_2(M;\mathbb{Z}[\pi_X])}\cong\mathbb{Z}[\pi_X]$
or $H_2(M;\mathbb{Z}[\pi_X])$ is a cyclic torsion module with a short free resolution, 
and $H_2(M;\mathbb{Z}[\pi_X])=0$.
In either case $H_1(M;\mathbb{Z}[\pi_X])$ is a cyclic module of projective dimension $\leq1$.

A similar argument applies for the pair $(Y,M)$.
\end{proof}

To use Theorem \ref{Z^2} to show that some $M$ has no abelian embedding
we must consider all possible bases for $Hom(\pi,\mathbb{Z})$,
or, equivalently, for $\pi/\pi'$.

\begin{ex}
Let $L$ be the link obtained from the Whitehead link $Wh=5^2_1$ 
by tying a reef knot ($3_1\#\!-\!3_1$) 
in one component. Then no embedding of $M(L)$ is abelian.
\end{ex}

The link group $\pi{L}$ has the presentation
\[
\langle{a,b,c,r,s,t,u,v,w}\mid
{as^{-1}vsa^{-1}=w=brb^{-1}},~cac^{-1}=b,~rcr^{-1}=a,~wcw^{-1}=b,
\]
\[
rvr^{-1}=tut^{-1},~sts^{-1}=u,~usu^{-1}=t,~vsv^{-1}=r\rangle,
\]
and $\pi_1(M(L))\cong\pi{L}/\langle\langle\lambda_a,\lambda_r\rangle\rangle$,
where $\lambda_a=c^{-1}wr^{-1}a$ and $\lambda_r=vu^{-1}s^{-1}t^{-1}rsa^{-1}b$
are the longitudes of $L$.
Let $b=\beta{a}$, $c=\gamma{a}$ and $t=r\tau$. 
Then $w=\gamma{r}$ in $\pi=\pi_1(M(L))$, and so  $\pi$ has the presentation
\[
\langle{a,\beta,\gamma,r,s,\tau,v}\mid
[r,a]=\gamma^{-1}\beta{r}\beta^{-1}r^{-1}=r\gamma^{-1}r^{-1},~
\gamma{a}\gamma^{-1}a^{-1}=\beta,~sr\tau{s}=r\tau{sr\tau},
\]
\[
as^{-1}vsa^{-1}=\gamma{r},~vs=rv,~
v=\tau{s}r\tau{s^{-1}}\tau^{-1}=\beta^{-1}s^{-1}\tau{s^2}r\tau{s^{-1}}\rangle.
\]
Now let $s=\sigma{r}$ and $v=\xi{r}$.
Then $\pi/\pi''$ has the metabelian presentation
\[
\langle{a,\beta,\gamma,r,\sigma,\tau,\xi}\mid
[r,a]=\gamma^{-1}\beta.{r}\beta^{-1}r^{-1}=r\gamma^{-1}r^{-1},~\gamma.{a}\gamma^{-1}a^{-1}=\beta,~
\]
\[
r^{-1}\sigma{r}.r\tau\sigma{r^{-1}}=\tau\sigma.{r^2}\tau{r^{-2}},
~ar^{-1}\sigma^{-1}\xi{ra^{-1}}.a\sigma{a^{-1}}=\gamma.r\gamma^{-1}r^{-1},~\xi={r}\xi\sigma^{-1}r^{-1},
\]
\[
\xi=\tau\sigma.{r^2}\tau{r^{-2}}.r\sigma^{-1}\tau^{-1}r^{-1}=
\beta^{-1}.r^{-1}\tau{r}.\sigma.{r^2}\tau{r^{-2}}.r\sigma^{-1}r^{-1},~[[\,,\,],[\,,\,]]=1\rangle,
\]
in which $\beta,\gamma,\sigma,\tau$ and $\xi$ represent elements of $\pi'$, 
which is the normal closure of the images of these generators.
The first relation expresses the commutator $[r,a]$ 
as a product of conjugates of these generators.
Using the third relation to eliminate $\beta$, 
we see that $\pi'/\pi''$ is generated as a module over 
$\mathbb{Z}[\pi/\pi']=\mathbb{Z}[a^\pm,r^\pm]$ by the images of 
$\gamma,\sigma,\tau$ and $\xi$, with the relations
\[
(1-r)[\gamma]=0,
\]
\[
(r^2-r+1)[\sigma]=r(r^2-r+1)[\tau]=0,
\]
\[
[\xi]=(1-r)[\sigma],
\]
and 
\[
2[\sigma]+2[\tau]=(a-1)[\gamma].
\]
If we extend coefficients to the rationals to simplify the analysis, 
we see that $P=H_1(M;\mathbb{Q}[\pi/\pi'])=\mathbb{Q}\otimes\pi'/\pi''$
is generated by $[\gamma]$ and $[\tau]$,
with the relations 
\[(1-r)[\gamma]=(r^2-r+1)[\tau]=0.
\]
Let $\{x,y\}$ be a basis for $\pi/\pi'$. Then $x=a^mr^n$ and $y=a^pr^q$,
where $|mq-np|=1$.
Let $\{x^*,y^*\}$ be the Kronecker dual basis for $Hom(\pi,\mathbb{Z})$,
and let $M_x$ and $M_y$ be the infinite cyclic covering spaces corresponding to
$\mathrm{Ker}(x^*)$ and $\mathrm{Ker}(y^*)$, respectively.
Then $H_1(M_x;\mathbb{Q})\cong{(P/(y-1)P\oplus\langle{y}\rangle)/(x.y=y+[x,y])}$.
If this module is cyclic as a module over the PID $\mathbb{Q}[x,x^{-1}]$ then so is 
the submodule 
\[
P/(y-1)P\cong\mathbb{Q}[\pi/\pi']/(r^2-r+1,y-1)\oplus\mathbb{Q}[\pi/\pi']/(r-1,y-1).
\]
On substituting $y=a^pr^q$ we find that this is so if and only if $p=0$ and $q=\pm1$.
But then $x=a^{\pm1}$, and a similar calculation show that 
$H_1(M_y;\mathbb{Q})$ is not cyclic  as a $\mathbb{Q}[y,y^{-1}]$-module.
Thus no basis for $\pi/\pi'$ satisfies the criterion of Theorem \ref{Z^2}, 
and $M$ has no abelian embedding.

We shall assume henceforth that $M=M(L)$, 
where $L$ is a 2-component link with
components slice knots and linking number $\ell=0$.
Let $x$ and $y$ be the images of the meridians of $L$ in $\pi$, 
and let $D_x$ and $D_y$ be slice discs for the components of $L$, 
embedded on opposite sides of the equator $S^3\subset{S^4}$.
Then the complementary regions for the embedding $j_L$ determined by $L$ are
$X_L=(D^4\setminus{N(D_x)})\cup{D_y\times{D^2}}$ and
$Y_L=(D^4\setminus{N(D_y)})\cup{D_x\times{D^2}}$.
The kernels of the natural homomorphisms from $\pi$ to $\pi_{X_L}$ 
and $\pi_{Y_L}$ are the normal closures of $y$ and $x$, respectively.
If one of the components of $L$ is unknotted then 
the corresponding complementary region is a handlebody of the form $S^1\times{D^3}\cup{h^2}$.
Inverting the handle structure gives a handlebody structure
$M\times[0,1]\cup{h^2}\cup{h^3}\cup{h^4}$.

If the components of $L$ are unknotted then $j_L$ is abelian,
and $\pi_X\cong\pi_Y\cong\mathbb{Z}$.

If $L$ is interchangeable there is a self-homeomorphism of $M(L)$ 
which swaps the meridians.
Hence $X_L$ is homeomorphic to $Y_L$, and $S^4$ is a twisted double.

The next result  has fairly strong hypotheses, 
but we shall give an example after the theorem showing that some such hypotheses are necessary.

\begin{theorem}
\label{Z^2rel}
Let $M$ be a $3$-manifold with fundamental group $\pi$
such that ${\pi/\pi'\cong\mathbb{Z}^2}$\!,
and suppose that $j_1$ and $j_2$ are abelian embeddings of $M$ in $S^4$.
If $(X_1,M)\simeq(X_2,M)$ and $(Y_1,M)\simeq(Y_2,M)$ then 
$j_1$ and $j_2$ are equivalent.
\end{theorem}

\begin{proof}
As in Theorem \ref{homhandle}, since $L_5(\mathbb{Z})$ acts trivially on the structure sets
there are homeomorphisms $F:X_1\to{X_2}$ and $G:Y_1\to{Y_2}$
which agree on $M$.
The union $h=F\cup{G}$ is a homeomorphism such that $hj_1=j_2$.
\end{proof}

To find examples where the complementary regions are {\it not\/} 
homeomorphic we should start with a link $L$ which is not interchangeable.
The simplest condition that ensures that a link with unknotted components 
is not interchangeable is asymmetry of the Alexander polynomial,
and the smallest such link with linking number 0 is $8^2_{13}$.
Since $\pi=\pi_1(M)$ is a quotient of $\pi{L}$, 
there remains something to be checked.

\begin{ex}
The complementary regions of the embedding of  $M(8^2_{13})$
determined by  the link $L=8^2_{13}$ are not homeomorphic
(although they are homotopy equivalent).
\end{ex}

Let $M=M(L)$.
The link group $\pi{L}=\pi8^2_{13}$ has the presentation
\[
\langle{s,t,u,v,w,x,y,z}\mid{yv=wy}, ~zx=wz,~ty=zt,~uy=zu,~sv=us,~vs=xv,
\]
\[
wu=tw,~xs=tx\rangle
\]
and the longitudes are $u^{-1}t$ and $x^2z^{-1}ys^{-1}w^{-1}xv^{-1}$.
Hence $\pi=\pi_1(M)$ has the presentation
\[
\langle{s,t,v,w,x,y}\mid{yv=wy}, ~tyt^{-1}x=wtyt^{-1}\!,~x^2ty^{-1}t^{-1}ys^{-1}w^{-1}xv^{-1}=1,
\]
\[sv=ts,~vs=xv,~wt=tw,~xs=tx\rangle.
\]
Setting $s=x\alpha$, $t=x\beta$, $v=x\gamma$ and $w=x\delta$, we obtain the presentation
\[
\langle{\alpha,\beta,\gamma,\delta,x,y}\mid
{[x,y]=xy\gamma(xy)^{-1}}\!.x\delta{x^{-1}},~
\beta.{y}\beta{y^{-1}}=\delta.{x}\beta{x}^{-1}\!.xy\beta^{-1}(xy)^{-1}\!.[x,y],
\]
\[
x^2\beta{x^{-2}}\!.x^2y^{-1}\beta^{-1}yx^{-2}=\gamma\delta.x\alpha{x^{-1}}.xy^{-1}[x,y]^{-1}yx^{-1}~
\]
\[
\delta{x}\beta=\beta{x}\delta,~\alpha{x}\gamma=\beta{x}\alpha,~\gamma{x}\alpha=x\gamma,~x\alpha=\beta{x}
\rangle
\]
in which $\alpha,\beta,\gamma$ and $\delta$ represent elements of $\pi'$, 
which is the normal closure of the images of these generators.
The subquotient $\pi'/\pi''$ is generated as a module over 
$\mathbb{Z}[\pi/\pi']\cong\Lambda_2=\mathbb{Z}[x^\pm,y^\pm]$
by the images of $\gamma$ and $\delta$,
with the relations 
\[
(x+1)(y-1)(x-1)[\gamma]=xy[\gamma]-x[\delta],
\]
\[(x-1)^2[\gamma]=(x-1)[\delta],
\]
and
\[
(x^2-x+1)[\gamma]=0,
\]
since $[\alpha]=x^{-1}(x-1)[\gamma]$ and $[\beta]=(x-1)[\gamma]$.
Adding the first two relations and rearranging gives
\[
[\delta]=-((x^2-x+1)y+2-2x) [\gamma]=2(x-1)[\gamma].
\]
Hence $\pi'/\pi''\cong\Lambda_2/( x^2-x+1,3(x-1)^2)=\Lambda_2/(x^2-x+1,3)$.
As a module over the subring $\mathbb{Z}[x,x^{-1}]$,
this is infinitely generated, but as a module over 
$\mathbb{Z}[y,y^{-1}]$ it has two generators.
Therefore there is no automorphism of $\pi$ which  induces an isomorphism 
$\mathrm{Ker}(\pi_1(j_X))=\pi'\rtimes\langle{x}\rangle
\cong\mathrm{Ker}(\pi_1(j_Y))=\pi'\rtimes\langle{y}\rangle$.
Hence $(X,M)$ and $(Y,M)$ are not homotopy equivalent as pairs, 
although $X\simeq{Y}$.

Does $M$ have any other abelian embeddings 
with neither complementary component homeomorphic to $X$,
perhaps corresponding to distinct link presentations?
Is this 3-manifold homeomorphic to a 3-manifold $M(\tilde{L})$ 
via a homeomorphism which does not preserve the meridians?

There is just one 3-manifold with $\pi$ elementary amenable and
$\beta=2$ which embeds in $S^4$ \cite{CH}.
This is the $\mathbb{N}il^3$-manifold $M=M(1;(1,1))$,
and $\pi=\pi_1(M)$ is the free nilpotent group of class 2 on 2-generators:
$\pi\cong{F(2)}/F(2)_{[3]}$.
This manifold may be obtained by 0-framed surgery on the Whitehead link $Wh=5^2_1$,
and the corresponding embedding is abelian.

All epimorphisms from $F(2)/F(2)_{[3]}$ to $\mathbb{Z}$ are equivalent 
under composition with automorphisms,
and each automorphism of $F(2)/F(2)_{[3]}$ is induced by a self-diffeomorphism of $M$.
If $j$ is an abelian embedding such that 
$(X,M)$ and $(Y,M)$ are homotopy
equivalent ({\it rel\/} $M$) to $(X_{Wh},M)$,
then $j$ is equivalent to $j_{Wh}$,
and the two complementary regions are homeomorphic.
However,
since $X$ and $Y$ are not aspherical, Lemma \ref{asph} does not apply
to provide a homotopy equivalence of pairs.
Is $j_{Wh}$ essentially unique?

\section{the higher rank cases}

Theorems \ref{neutral} and \ref{Z^2} have analogues when $\beta=3$, 4 or 6.

\begin{theorem}
\label{b=3}
Let $M$ be a $3$-manifold with fundamental group $\pi$ 
such that $\pi/\pi'\cong\mathbb{Z}^3$.
If $j:M\to{S^4}$ is an abelian embedding then $X\simeq{T}$ 
and ${Y}\simeq{S^1\vee2{S^2}}$,
while  $H_1(M;\mathbb{Z}[\pi_X])\cong\mathbb{Z}$ and 
$H_1(M;\mathbb{Z}[\pi_Y])$ is a torsion $\mathbb{Z}[\pi_Y]$-module
of projective dimension $1$ and which can be generated by two elements.
The component $X$ is determined up to homeomorphism by its boundary $M$,
while $Y$ is determined by the homotopy type of the pair $(Y,M)$.
\end{theorem}

\begin{proof}
The classifying map $c_X:X\to{K(\pi_X,1)}\simeq{T}$ is a homotopy equivalence,
by Theorem \ref{2con},  since $c.d.X=c.d.T=2$ and $\chi(X)=\chi(T)=0$.
The equivalence ${Y}\simeq{S^1\vee2{S^2}}$ follows from Corollary \ref{2concor},
since $\pi_Y\cong\mathbb{Z}$ and $\chi(Y)=2$.

Since $H_2(X;\mathbb{Z}[\pi_X])=0$,
the exact sequence of homology for the pair $(X,M)$ 
with coefficients $\mathbb{Z}[\pi_X]$ reduces to an isomorphism
$H_1(M;\mathbb{Z}[\pi_X])\cong
\overline{H^2(X;\mathbb{Z}[\pi_X])}\cong\mathbb{Z}$.

Similarly, there is an exact sequence
\[
0\to\mathbb{Z}\to{H_2(M;\mathbb{Z}[\pi_Y])}\to\mathbb{Z}[\pi_Y]^2\to\mathbb{Z}[\pi_Y]^2\to
{H_1(M;\mathbb{Z}[\pi_Y])}\to0,
\]
since $H_2(Y;\mathbb{Z}[\pi_Y])\cong\mathbb{Z}[\pi_Y]^2$ and 
$H_2(Y,M;\mathbb{Z}[\pi_Y])\cong\overline{H^2(Y;\mathbb{Z}[\pi_Y]) }$.
Let $A=\pi'/\pi''$, considered as a $\mathbb{Z}[\pi/\pi']$-module.
Then $A$ is finitely generated as a module, since $\mathbb{Z}[\pi/\pi']$ is a noetherian ring.
Let $\{x,y,z\}$ be a basis for $\pi/\pi'$ such that $j_{X*}(y)=0$ and $j_{Y*}(x)=j_{Y*}(z)=0$.
Then $H_1(M;\mathbb{Z}[\pi_X])\cong\mathbb{Z}$ is an extension of $\mathbb{Z}$ by $A/(y-1)A$,
and so $A=(y-1)A$.
Similarly, $H_1(M;\mathbb{Z}[\pi_Y])$ is an extension of $\mathbb{Z}^2$ by $A/(x-1,z-1)A$.
Together these observations imply that $H_1(M;\mathbb{Z}[\pi_Y])$ is a torsion $\mathbb{Z}[\pi_Y]$-module,
and so the fourth homomorphism in the above sequence is a monomorphism.
Thus $H_1(M;\mathbb{Z}[\pi_Y])$ is a torsion $\mathbb{Z}[\pi_Y]$-module
with projective dimension $\leq1$,
and is clearly generated by two elements.
(Note also that a torsion $\mathbb{Z}[\pi_Y]$-module of projective dimension 0 is 0.)

Since $X$ is aspherical, Lemma \ref{asph} applies, 
and so the homotopy type of the pair $(X,M)$ is determined by $M$.
The final assertion follows (as in Theorem \ref{Z^2rel}),
since $L_5(\mathbb{Z}^2)$ and 
$L_5(\mathbb{Z})$ act trivially on the structure sets 
$\mathcal{S}^{TOP}(X,\partial{X})$ and $\mathcal{S}^{TOP}(Y,\partial{Y})$,
by  \cite[Theorem 6.7]{Hi}.
\end{proof}

The link $L=9^3_{21}$ has an unique partition as a bipartedly slice link,
and for the corresponding embedding $\pi_{X_L}\cong{F(2)}$
and $\pi_{Y_L}\cong\mathbb{Z}$.
Then $M=M(9^3_{21})\cong(S^2\times{S^1})\#M(5^2_1)$,
so $\pi\cong\mathbb{Z}*F(2)/F(2)_{[3]}$,
with presentation $\langle{x,y,z}\mid[x,y]\leftrightharpoons{x,y}\rangle$.
It is not hard to show that the kernel of any epimorphism
$\phi:\pi\to\langle{t}\rangle\cong\mathbb{Z}$ 
has rank $\geq1$ as a $\mathbb{Z}[t,t^{-1}]$-module.
Hence $M$ has no abelian embedding, by Theorem \ref{b=3}.

The 3-torus $T^3=\mathbb{R}^3/\mathbb{Z}^3$ has an abelian embedding,
as the boundary of a regular neighbourhood of an unknotted embedding of $T$ in $S^4$.
This manifold may be obtained by 0-framed surgery on the Borromean rings $Bo=6^3_2$,
and also  on $9^3_{18}$. 
The three bipartite partitions of $Bo$ lead to equivalent embeddings.
(However these are clearly not isotopic!)
The link $9^3_{18}$ has two bipartedly slice partitions (both bipartedly trivial).
Any such embedding of $T^3$ has 
$X\cong{T}\times{D^2}$ and $Y\simeq{S^1}\vee2S^2$.
Does $T^3$ have an essentially unique abelian embedding?

If $M$ is Seifert fibred and $\pi/\pi'\cong\mathbb{Z}^3$
then it has generalized Euler invariant $\varepsilon=0$, 
and so is a mapping torus $T_g\rtimes_\theta{S^1}$, 
with orientable base orbifold and monodromy $\theta$ of finite order.
Are there any such manifolds other than the 3-torus 
which have abelian embeddings?

Suppose that $\beta=3$ and $M$ has an embedding $j$ 
such that $H_1(Y;\mathbb{Z})=0$.
If $f:\pi\to\mathbb{Z}^2$ is an epimorphism  with kernel $\kappa$ 
and $R=\mathbb{Z}[\pi/\kappa]_{f(S)}$ then $H_1(M;R)$ has rank 2, 
by Lemma \ref{completion},
and so the condition of Theorem \ref{b=3} does not hold.
Therefore no such 3-manifold can also have an abelian embedding. 

\begin{theorem}
\label{b=4}
Let $M$ be a $3$-manifold with fundamental group $\pi$ 
such that $\pi/\pi'\cong\mathbb{Z}^4$.
If $j:M\to{S^4}$ is an abelian embedding
then $X\simeq{Y}\simeq{T\vee{S^2}}$.
Hence  $H_1(M;\mathbb{Z}[\pi_X])$ is a quotient of
$\mathbb{Z}[\pi_X]\oplus\mathbb{Z}$  by a cyclic submodule
(and similarly for $H_1(M;\mathbb{Z}[\pi_Y])$).
The components $X$ and $Y$ are determined by the homotopy types 
of the pairs $(X,M)$ and $(Y,M)$, respectively.
\end{theorem}

\begin{proof}
As in Theorem \ref{Z^2}, generators
for $\pi_X\cong\mathbb{Z}^2$ and $\pi_2(X)\cong\mathbb{Z}[\pi_X]$ 
determine a map from $T^{[1]}\vee{S^2}$ to $X$.
This extends to a 2-connected map from $T\vee{S^2}$ to $X$,
which is a homotopy equivalence by Theorem \ref{2con}.
Hence $X\simeq{T\vee{S^2}}$.

The second assertion follows from the exact sequence 
of homology for $(X,M)$ with coefficients $\mathbb{Z}[\pi_X]$,
since $H^2(X;\mathbb{Z}[\pi_X])\cong{\mathbb{Z}[\pi_X]\oplus\mathbb{Z}}$.
Parallel arguments apply for $Y$ and $H_1(M;\mathbb{Z}[\pi_Y])$.

The final assertion follows as in Theorems \ref{Z^2rel} and \ref{b=3}.
\end{proof}

The argument below for the final case ($\beta=6$) is adapted from Wall's proof 
that the $(n-1)$-skeleton of a $PD_n$-complex is essentially unique 
 \cite[Theorem 2.4]{Wa}.

\begin{theorem}
Let $M$ be a $3$-manifold with fundamental group $\pi$ 
such that $\pi/\pi'\cong\mathbb{Z}^6$.
If $j:M\to{S^4}$ is an abelian embedding
then $X\simeq{Y}\simeq{T^{3[2]}}$, 
the $2$-skeleton of the $3$-torus $T^3$,
while  $H_1(M;\mathbb{Z}[\pi_X])$ and $H_1(M;\mathbb{Z}[\pi_Y])$ 
are cyclic $\mathbb{Z}[\pi_X]$- and $\mathbb{Z}[\pi_Y]$-modules
(respectively), of projective dimension $\leq1$.
\end{theorem}

\begin{proof}
Since $\beta=6$ and $j$ is abelian we may identify $\pi_X$ with $\mathbb{Z}^3$.
The first part of the argument of Theorem \ref{Z^2} applies to show that
$\pi_2(X)$ is isomorphic to $\Lambda_3=\mathbb{Z}[\mathbb{Z}^3]$.
Let $C_*$ and $D_*$ be the equivariant chain complexes
of the universal covers of $T^{3[2]}$ and $X$, respectively.
Since these are partial resolutions of $\mathbb{Z}$ there is
a chain map $f_*:C_*\to{D_*}$ such that $H_0(f)$ is an isomorphism.
Clearly $H_1(f)$ is also an isomorphism.
We shall modify our choice of $f_*$ so that it is a chain homotopy equivalence.

The $\Lambda_3$-modules $H_2(C_*)<{C_2}$ and $H_2(D_*)<D_2$ are free of rank 1.
Let $t\in{C_2}$ and $x\in{D_2}$ represent generators for these submodules,
and let $t^*$ and $x^*$ be the Kronecker dual generators of the cohomology modules 
$H^2(C^*)=Hom(H_2(C_*),\Lambda_3)\cong\Lambda_3$  and 
$H^2(D^*)=Hom(H_2(D_*),\Lambda_3)\cong\Lambda_3$, respectively.
Let $f'_i=f_i$ for $i=0,1$, and  let $f_2'(u)=f_2(u)-z^*(u)x$ for all $u\in{C_2}$,
where $z^*=H^2(f^*)(x^*)-t^*\in{Hom(H_2(C_*),\Lambda_3)}$.
Then $f'_*$ is again a chain homomorphism, and $H_2(f'_*)$ is an isomorphism.
Hence $f_*$ is a chain homotopy equivalence. 
This may be realized by a map from $T^{3[2]}$ to the 2-skeleton $X^{[2]}$,
and the composite with the inclusion $X^{[2]}\subseteq{X}$ is then
a homotopy equivalence.

A similar argument applies for $Y$.
The second assertion follows as before.
\end{proof}

In this case the natural transformation $I_G:G\to{L_5^s(G)}$ 
used  \cite[Theorem 6.7]{Hi} maps $G=\mathbb{Z}^3$ onto a direct summand of index 2
in $L_5(\mathbb{Z}^3)$,
and it is no longer clear that $X$ and $Y$ are determined by the homotopy types of 
the pairs $(X,M)$ and $(Y,M)$.

Lemma \ref{completion} and Theorem \ref{b=3} again imply that 
when $\beta=4$ or 6 no 3-manifold which has an embedding $j$ 
such that $H_1(Y;\mathbb{Z})=0$ 
can also have an abelian embedding.
However, if $L$ is the 4-component link obtained from $Bo$ by adjoining a parallel of one component,
then $M(L)$ has an abelian embedding with $X\cong{Y}$ and $\chi(X)=1$, 
and also has an embedding with $\chi(X)=-1$.
We shall not give more details,
as no natural examples demand our attention in these cases.

\section{ 2-component links with $\ell\not=0$}

If $M$ is a rational homology sphere with an abelian embedding then
$\pi/\pi'\cong(\mathbb{Z}/\ell\mathbb{Z})^2$ and $\pi_X\cong\pi_Y\cong\mathbb{Z}/\ell\mathbb{Z}$, 
for some $\ell\not=0$.
In particular, if  $L$ is a 2-component link with linking number
$\ell\not=0$ then $M(L)$ is a rational homology sphere, 
and if the components of $L$ are unknotted then $j_L$ is abelian.
Six of the eight  rational homology 3-spheres with elementary amenable groups 
and which embed in $S^4$ have such link presentations,
with $\ell\leq4$.
(In particular, $S^3=M(2^2_1)$, where $2^2_1$ is the Hopf link!)
Since $M$ is an integral homology 3-sphere if $\ell=1$,
we may assume that $\ell>1$.
The simplest such link is $L=2\ell^2_1$, the $(2,2\ell)$-torus link,
for which $M(L)\cong{M(0;(\ell,1),(\ell,1),(\ell,-1))}$. 
Since $L$ is interchangeable, $X_L\cong{Y_L}$.
(An argument based on explicit embeddings of $P_\ell$ in $S^4$
is given in \cite{Ya}.)

There is again a necessary condition for the existence of such an embedding.

\begin{lemma}
\label{torsion}
Let $M$ be a $3$-manifold with fundamental group $\pi$ 
such that $\pi/\pi'\cong(\mathbb{Z}/\ell\mathbb{Z})^2$, for some $\ell\not=0$.
If $j:M\to{S^4}$ is an abelian embedding then $X\simeq{Y}\simeq{P_\ell}$,
and $H_1(M;\mathbb{Z}[\pi_X])$ and $H_1(M;\mathbb{Z}[\pi_Y])$ 
are cyclic $\mathbb{Z}[\pi_X]$- and $\mathbb{Z}[\pi_Y]$-modules (respectively), 
and are quotients of $\mathbb{Z}^{\ell-1}$, as abelian groups.
\end{lemma}

\begin{proof}
The first assertion holds by Lemma \ref{finite}.
The second part then follows from the exact sequences of homology for the 
universal covering spaces of the pairs $(X,M)$ and $(Y,M)$,
since $\widetilde{X}\simeq\widetilde{Y}\simeq\vee^{\ell-1}S^2$.
(Note that $H_2(\widetilde{X};\mathbb{Z})$ and $H^2(\widetilde{X};\mathbb{Z})$
are each isomorphic to the augmentation ideal of $\mathbb{Z}[\pi_X]$,
which is cyclic as a module and free of rank $\ell-1$ as an abelian group.)
\end{proof}

To use Theorem \ref{torsion} to show that some $M$ has no abelian embedding
we must consider all possible bases for $Hom(\pi,\mathbb{Z}/\ell\mathbb{Z})$,
or, equivalently, for $\pi/\pi'$.

When $\ell=2$, 
we have $X\simeq{Y}\simeq{RP^2}$,
and the composite $\partial\widetilde{X}\subset\widetilde{X}\simeq{S^2}$ induces an isomorphism on $H_2$.
There are two homotopy classes of maps $\partial\widetilde{X}\to{S^2}$
inducing each generator of $H^2(\partial\widetilde{X};\mathbb{Z})$,
by  \cite[Theorem 8.4.11]{Sp}.
It follows that the homotopy type of the pair $(X,M)$ is determined up to a finite ambiguity by $M=\partial{X}$.
The structure set  $\mathcal{S}_{TOP}(X,M)$
has two elements, since $L_5(\mathbb{Z}/2\mathbb{Z})=0$.
We may conclude that if $\pi/\pi'\cong(\mathbb{Z}/2\mathbb{Z})^2$ 
then $M$ has only finitely many abelian embeddings.

The quaternion manifold $M=S^3/Q(8)=M(4^2_1)$
has an essentially unique abelian embedding.
The complementary regions are homeomorphic to the total space $N$ 
of the disc bundle over $RP^2$ with Euler number 2 \cite{La}.
Lawson constructed a self-homotopy equivalence of $N$ 
which is the identity outside a regular neighbourhood of an essential $S^1$,
and which has nontrivial normal invariant.
His  construction extends to all $X\simeq{RP^2}$.
Do all the resulting self-homotopy equivalences have nontrivial normal invariant?

The links $9^2_{38}$, $9^2_{57}$ and $9^2_{58}$
each have unknotted components, 
asymmetric Alexander polynomial and linking number 2.
They are candidates for examples with $X$ and $Y$ not homeomorphic.

Let $L$ be the link obtained by tying a slice knot 
with non-trivial Alexander polynomial 
(such as the stevedore's knot $6_1$) in one component of $4^2_1$.
Then $M(L)$ embeds in $S^4$, but does not satisfy Lemma \ref{torsion},
since for two of the three 2-fold covers of $M(L)$ the first homology 
is not cyclic as an abelian group.
Hence $M(L)$ has no abelian embedding.

Suppose next that $\ell=3$.
The manifold $M(6^2_1)$ is a $\mathbb{N}il^3$-manifold
with Seifert base the flat orbifold $S(3,3,3)$,
and $X_L\cong{Y_L}$.
Since $Wh(\mathbb{Z}/3\mathbb{Z})=0$ and $L_5(\mathbb{Z}/\ell\mathbb{Z})=0$
for $\ell$ odd,
the pair $(X,M)$ given by an abelian embedding is determined 
up to homeomorphism by its homotopy type {\it  rel} $\partial$.

The most interesting example with $\ell=4$ is perhaps $M(8^2_2)\cong{M(-1;(2,1),(2,3))}$,
which is a $\mathbb{N}il^3$-manifold. 
The link $8^2_2$ is not a torus link, but is interchangeable, and so $X_L\cong{Y_L}$. 

Each of the links $9^2_{53}$ and $9^2_{61}$ has  
unknotted components and $\ell=4$,
and gives a $\mathbb{S}ol^3$-manifold with an abelian embedding. 
Is either of these links interchangeable?

There remains one more $\mathbb{S}ol^3$-manifold which embeds in $S^4$.
This is $M_{2,4}=N\cup_\phi{N}$, 
where $N$ is the mapping cylinder of the orientation cover of the Klein bottle and 
$\pi=\left(\begin{smallmatrix}
2&-9\\1&-4
\end{smallmatrix}\right)$.
The group $\pi=\pi_1(M_{2,4})$ has a presentation
\[
\langle{u,v,x,y}\mid{uvu^{-1}=v^{-1},~xyx^{-1}=y^{-1}},~x^2=u^4v^{-9},~y=u^2v^{-4}\rangle.
\]
(Hence also $u^2=x^{-8}y^9$ and $v=x^{-2}y^2$, 
and $\langle{u,v}\rangle=\langle{x^2,y}\rangle\cong\mathbb{Z}^2$ .)
Solving for $y$, and setting $v=x^2w$, we get the presentation
\[
\langle{u,w,x}\mid{xu^2=(x^2w)^4u^{-2}x(x^2w)^4},~u=x^2wux^2w,~x^2(x^2w)^9=u^4\rangle.
\]
This manifold also arises from surgery on the link $L=(U,8_{20})$ of  \cite[Figure 1]{CH}, 
with components the unknot $U$ and the slice knot $8_{20}$,  and with $\ell=4$.
Let  $D\subset{D^4}$ be the slice disc for the knot $8_{20}$
obtained by desingularizing the ribbon disc visible in the right-hand part of  \cite[Figure 1]{CH}.
Let $X_L$ be the region obtained from $D^4$ by deleting a regular neighbourhood of  $D$ 
and adding a 2-handle along the unknotted component $U$,
and let $Y_L$ be the complementary region.
Then $\pi_{X_L}$ is a quotient of the ``ribbon group" $\pi_1(D^4\setminus{D})$.
Using \cite[Theorem 1.15]{HiA}, it may shown that 
$\pi_{X_L}\cong\mathbb{Z}/3\mathbb{Z}\rtimes_{-1}\mathbb{Z}/4\mathbb{Z}$.
On the other hand,  $\pi_{Y_L}\cong\mathbb{Z}/4\mathbb{Z}$.

Let $\lambda_{i,j}:\pi\to\mathbb{Z}/4\mathbb{Z}$ be the epimorphism
sending $x, u$ to $i,j\in\mathbb{Z}/4\mathbb{Z}$, respectively.
It can be shown that the abelianization of $\mathrm{Ker}(\lambda_{i,j})$ 
is a quotient of the augmentation ideal in $\mathbb{Z}[\mathbb{Z}/4\mathbb{Z}]$, for 
$(i,j)=(1,0)$ or (2,1). Since these epimorphisms form a basis 
for $Hom(\pi,\mathbb{Z}/4\mathbb{Z})$, 
we cannot use Lemma \ref{torsion} to rule out an abelian embedding for $M_{2,4}$.
Is there a 2-component link with unknotted components which gives rise to this manifold?

\section{some remarks on the mixed cases}

If $M$ is a 3-manifold with an abelian embedding such that 
$\pi_X\cong{G_k}=\mathbb{Z}\oplus(\mathbb{Z}/k\mathbb{Z})$,
for some $k>1$, then $\pi_Y\cong{G_k}$ also,
and so $H_1(M;\mathbb{Z})\cong\mathbb{Z}^2\oplus(\mathbb{Z}/k\mathbb{Z})^2$,
which requires four generators. 
The simplest examples may be constructed from 4-component links
obtained by replacing one component of the Borromean rings $Bo$ by its $(2k,2)$ cable.

In this case even the determination of the homotopy types of the complements is not clear.
The group $G_k$ has minimal presentations
\[
\mathcal{P}_{k,n}=\langle{a,t}\mid{a^k,~ta^n=a^nt}\rangle,
\]
where $0<n<k$ and $(n,k)=1$.
The associated 2-complexes $S_{k,n}=S^1\vee{P_k}\cup_{[t,a^n]}e^2$ have Euler characteristic 1,
and it is easy to see that there are maps between them which induce isomorphisms on fundamental groups.
We may identify $S_{k,n}$ with $T\cup{MC}\cup{P_k}$, 
where $MC$ is the mapping cylinder of the degree-$n$ map 
$z\mapsto{z^n}$ from $\{1\}\times{S^1}\subset{T}$ to the 1-skeleton $S^1\subset{P_k}$.
In particular,  $S_k=S_{k,1}=T\cup_{a^k}e^2$ is  the 2-skeleton of $S^1\times{P_k}$.
From these descriptions it is easy to see that 
(1) automorphisms of $G_k$ which fix the torsion subgroup $A=\langle{a}\rangle$ may be realized 
by self homeomorphisms  of $S_{k,n}$ which act by reflections and Dehn twists on $T$, and fix the second 2-cell;
and (2) the automorphism which fixes $t$ and inverts $a$ is induced by an involution of $S_{k,n}$.

Let $C(k,n)_*$ be the cellular chain complex of the universal cover of $S_{k,n}$.
A choice of basepoint for $S_{n,k}$ determines lifts of the cells of $S_{k,n}$,
and hence isomorphisms $C(k,n)_0\cong\Gamma$, $C(k,n)_1\cong\Gamma^2$
and $C(k,n)_2\cong\Gamma^2$. 
The differentials are given by $\partial_1=(a-1,t-1)$ and 
$\partial_2^n=\left(\begin{smallmatrix}
(t-1)\nu_n&\rho\\ 1-a&0
\end{smallmatrix}\right)$,
where $ \nu_n=\Sigma_{0\leq{i}<n}a^i$  and $\rho=\nu_k$.
Let $\{e_1,e_2\}$ be the standard basis for $C(k,n)_2$.
Then $\Pi_{k,n}=\pi_2(S_{k,n})=\mathrm{Ker}(\partial_2^n)$
is generated by $g=\rho{e_1}-n(t-1)e_2$ and $h=(a-1)e_2$,
with relations $(a-1)g=n(t-1)h$ and $\rho{h}=0$.
It can be shown that $\Pi_{k,n}\cong\alpha^*\Pi_{k,m}$,
where $\alpha$ is the automorphism of $G_k$ such that $\alpha(t)=t$
and $\alpha(a)=a^r$, where $n\equiv{rm}$ {\it mod} $k$.
Is there a chain homotopy equivalence $C(k,n)_*\simeq\alpha^*C(k,m)_*$?

Is every finite 2-complex $S$
with $\pi_1(S)\cong{G_k}$ and $\chi(S)=1$ 
homotopy equivalent to $S_{k,n}$, for some $n$?
The key invariants are the $\Gamma$-module $\pi_2(S)$ and the $k$-invariant in 
$H^3(G_k;\pi_2(S))$.
Let $S_{\langle{t}\rangle}$ be the finite covering space with fundamental group
${\langle{t}\rangle}\cong\mathbb{Z}$.
If $M$ is a finitely generated submodule of a free $\Gamma$-module
then $H^i(\langle{t}\rangle;M)=0$ for $i\not=1$, 
while $H^1(\langle{t}\rangle;M)=M_t=M/(t-1)M$.
Hence the spectral sequence
\[
H^p(A;(H^q(\langle{t}\rangle;M))\Rightarrow{H^{p+q}(G_k;M)}
\] 
collapses,
to give $H^{p+1}(G_k;M)\cong{H^p(A;M_t)}$.
If $M=\pi_2(S)$ then $M_t\cong{H_2(S_{\langle{t}\rangle}};\mathbb{Z})$,
as a $\mathbb{Z}[A]$-module. 
When $M=\Pi_{k,n}$ it is easy to see that $M_t\cong\mathbb{Z}\oplus{I_A}$,
where $I_A$ is the augmentation ideal of $\mathbb{Z}[A]$,
and so $H^2(A;M_t)\cong{H^2(A;\mathbb{Z})}\cong\mathbb{Z}/k\mathbb{Z}$.

Let $V$ and $W$ be finite 2-complexes with $\pi_1(V)\cong\pi_1(W)\cong{G_k}$,
and let $\Gamma=\mathbb{Z}[G_k]$.
Then $\chi(V)\geq1$ and $\chi(W)\geq1$, and an application of Schanuel's Lemma
to the chain complexes of the universal covers gives
\[
\pi_2(V)\oplus\Gamma^{\chi(W)}\cong\pi_2(W)\oplus\Gamma^{\chi(V)}.
\]
Taking $W=S_{k,1}$, we see that $H^3(G;\pi_2(V))\cong\mathbb{Z}/k\mathbb{Z}$,
for all such $V$.

Even if we can determine the homotopy types of the 2-complexes $S$ with 
$\pi_1(S)$ and $\chi(S)=1$, and the homotopy types of the pairs $(X,M)$
for a given $M$, 
the groups $L_5^s(G)$ are commensurable 
with $L_4(\mathbb{Z}/k\mathbb{Z})$, which has rank $\lfloor\frac{k+1}2\rfloor$,
and so characterizing such abelian embeddings up to isotopy may be difficult.

The $S^1$-bundle spaces $M(-2;(1,0))$ (the half-turn flat 3-manifold $G_2$), 
and $M(-2;(1,4))$ (a $\mathbb{N}il^3$-manifold) do not have abelian embeddings, 
since $\beta=1$ but $\pi/\pi'$ has nontrivial torsion.
In each case $\pi$ requires 3 generators, and so
they cannot be obtained by surgery on a 2-component link.
However, they may be obtained by 0-framed surgery on the links $8^3_9$ and $9^3_{19}$,
respectively.
For the embeddings defined by these links $X\simeq{Kb}$ and $\pi_Y=\mathbb{Z}/2\mathbb{Z}$.
As in Theorem \ref{b=3}, 
$X$ is homeomorphic to the corresponding disc bundle space,
since Lemma \ref{asph} applies, and $L_5(\mathbb{Z}\rtimes_{-1}\!\mathbb{Z})$ 
acts trivially on the structure set $\mathcal{S}_{TOP}(X,\partial{X})$,
by \cite[Theorem 6.7]{Hi}.
(See also Theorem \ref{mtor} above.)
As discussed in \S1, $Y$ is homotopy equivalent to a finite 2-complex, and hence $Y\simeq{RP^2}\vee{S^2}$.
Are the corresponding embeddings of $Kb$ unknotted?

It is easy to find 3-component bipartedly trivial links $L$ such that $X_L$ 
is aspherical and $\pi_{X_L}$ is a solvable Baumslag-Solitar group $\mathbb{Z}*_m$.
Lemma \ref{asph} and surgery arguments again apply to show that $X_L$ is determined 
up to homeomorphism by $M$.
In this case $Y$ is homotopy equivalent to a finite 2-complex, by Lemma \ref{finite},
since $\pi_1(Y)\cong\mathbb{Z}/(m-1)\mathbb{Z}$,
$\chi(Y)=2$ and $c.d.Y\leq2$.
Hence $Y\simeq{P_{m-1}}\vee{S^2}$ \cite{DS}.

\smallskip
{\it Acknowledgment.} I would like to thank T.Abe, F.E.A.Johnson and J.L.Meiers for their responses
to my queries arising in the course of this work.


\end{document}